\documentclass[11pt]{article}
\usepackage{amssymb}
\usepackage{amsfonts}
\usepackage{amsmath}
\usepackage{amscd}
\usepackage{amsthm}
\usepackage{latexsym}
\usepackage[all]{xy}
\usepackage{CJK}
\usepackage{setspace}
\usepackage{indentfirst}
\usepackage{mathabx}
\numberwithin{equation}{section}
\date{}

\newtheorem{theorem}{Theorem}
\newtheorem{lemma}{Lemma}[section]
\newtheorem{thm}[lemma]{Theorem}
\newtheorem{definition}[lemma]{Definition}
\newtheorem{corollary}[lemma]{Corollary}
\newtheorem{example}[lemma]{Example}

\newcommand{\End}{\rm{End}}
\newcommand{\id}{\rm{id}}
\newcommand{\dom}{\rm{domdim}}
\newcommand{\mo}{\rm{mod}}
\newcommand{\Ext}{\rm{Ext}}
\newcommand{\Hom}{\rm{Hom}}
\newcommand{\Gpd}{\rm{Gpd}}
\newcommand{\pd}{\rm{pd}}
\newcommand{\soc}{\rm{soc}}
\newcommand{\rad}{\rm{rad}}
\newcommand{\im}{\rm{Im}}
\newcommand{\gl}{\rm{gldim}}

\begin{document}
\title{\bf Gorenstein projective dimensions of modules over minimal Auslander-Gorenstein algebras }
\author{Shen Li, \quad Ren\'{e} Marczinzik, \quad Shunhua Zhang }
\date{}
\maketitle

\begin{center}\section*{}\end{center}

\begin{abstract}

In this article we investigate the relations between the Gorenstein projective dimensions of $\Lambda$-modules and their socles for minimal n-Auslander-Gorenstein algebras $\Lambda$ in the sense of Iyama and Solberg \cite{IS}. First we give a description of projective-injective $\Lambda$-modules in terms of their socles. Then we prove that a $\Lambda$-module $N$ has Gorenstein projective dimension at most n iff its socle has Gorenstein projective dimension at most n iff $N$ is cogenerated by a projective $\Lambda$-module. Furthermore, we show that minimal n-Auslander-Gorenstein algebras can be characterised by the relations between the Gorenstein projective dimensions of modules and their socles.
\end{abstract}

{\bf Key words and phrases:} minimal Auslander-Gorenstein algebras, higher Auslander algebras, Gorenstein projective dimension

\vskip0.2in

\section{Introduction}

An Artin algebra is called an Auslander algebra if its global dimension is at most 2 and its dominant dimension is at least 2. Auslander established a bijection between representation-finite algebras and Auslander algebras, given by $M \mapsto {\End}_{A}M $ where $A$ is a representation-finite Artin algebra and $M$ is an additive generator of $A$. In \cite{IY1}, Iyama introduced n-Auslander algebras. As a generalisation of Auslander algebras, n-Auslander algebras are characterised by having global dimension at most n+1 and dominant dimension at least n+1. There is a one-to-one correspondence between n-Auslander algebras and finite n-cluster tilting subcategories, see \cite{IY1} for details. This is known as the higher Auslander correspondence. $\tau$-selfinjective algebras were introduced by Auslander and Solberg in \cite{AS}. The endomorphism algebra of some suitable module over a $\tau$-selfinjective algebra satisfies that the injective dimension of the left (right) regular module is at most 2 and the dominant dimension is at least 2. Thus it can be considered as a generalisation of Auslander algebras. There also exists a one-to-one correspondence between $\tau$-selfinjective algebras and algebras with injective dimension of the left (right) regular module at most 2 and dominant dimension at least 2, see \cite{AS} and \cite{KF} for details.

In \cite{IS}, Iyama and Solberg gave a further generalisation and defined the minimal n-Auslander-Gorenstein algebras. Then they introduced the notion of n-precluster tilting subcategories and established a one-to-one correspondence between minimal n-Auslander-Gorenstein algebras and finite n-precluster tilting subcategories. The n-precluster tilting subcategories generalize and unify two seemingly different concepts, namely n-cluster tilting subcategories and $\tau$-selfinjective algebras. We refer to \cite{IS} for details. See also \cite{Mar} for another generalisation and \cite{CIM} and \cite{Mar2} for examples and applications of minimal Auslander-Gorenstein algebras.

\begin{definition}

An Artin algebra $\Lambda$ is called a minimal n-Auslander-Gorenstein algebra if it satisfies \id$ _{\Lambda}\Lambda\leq n+1 \leq $ \dom \,$\Lambda$.

\end{definition}

Here id and domdim denote the injective dimension and dominant dimension, respectively. The definition of minimal n-Auslander-Gorenstein algebras is left-right symmetric, that is $\Lambda$ is a minimal n-Auslander-Gorenstein algebra if and only if $\Lambda^{op}$ is a minimal n-Auslander-Gorenstein algebra. Thus a minimal n-Auslander-Gorenstein algebra $\Lambda$ is either a self-injective algebra or an (n+1)-Gorenstein algebra satisfying ${\id} _{\Lambda}\Lambda = n+1 = $ \dom \,$\Lambda$. Let $N$ be a finitely generated $\Lambda$-module. Then the Gorenstein projective dimension of $N$ is at most n+1 for any Gorenstein algebra of self-injective dimension n+1. The main aim of this paper is to investigate the relations between the Gorenstein projective dimensions of $N$ and its socle for minimal n-Auslander-Gorenstein algebras. First we calculate the Gorenstein projective dimensions of all simple $\Lambda$-modules by using n-precluster tilting subcategories.

Let $A$ be an Artin algebra. We denote by $A$-mod the category of finitely generated left modules. In \cite{IY2} the functors$$\tau_{n}=\tau \Omega^{n-1}_{A}:A\text{-}\underline{\mo}\longrightarrow A\text{-}\overline{\mo} \quad\text{and}\quad \tau_{n}^{-}=\tau^{-}\Omega^{-(n-1)}_{A}:A\text{-}\overline{\mo}\longrightarrow A\text{-}\underline{\mo}$$are defined as the n-Auslander-Reiten translations.

\begin{definition}{\rm \cite[Definition 3.2]{IS}}

A subcategory $\mathcal{C}$ of $A$-{\mo} is called an n-precluster tilting subcategory if it satisfies the following conditions:

$(1)$\,$\mathcal{C}$ is a generator-cogenerator for $A$-{\mo},

$(2)$\,$\tau_{n}(\mathcal{C})\subseteq\mathcal{C}$ and $\tau^{-}_{n}\mathcal{(C)}\subseteq\mathcal{C}$,

$(3)$\,${\Ext}^{i}_{A}(\mathcal{C},\mathcal{C})=0$ for $0<i<n$,

$(4)$\,$\mathcal{C}$ is a functorially finite subcategory of $A$-{\mo}.

\end{definition}

If moreover $\mathcal{C}$ admits an additive generator $M$, we say that $\mathcal{C}$ is a finite n-precluster tilting subcategory. The endomorphism algebra $\Lambda={\End_{A}}(M)^{op}$ of $M$ is a minimal n-Auslander-Gorenstein algebra and all minimal n-Auslander-Gorenstein algebras can be constructed in this way. In particular, $\Lambda$ is self-injective if and only if $M$ is projective.

Let $M=\oplus^{t}_{i=1}M_{i}$ with $M_{i}$ indecomposable be a basic additive generator of an n-precluster tilting subcategory $\mathcal{C}$ in $A$-{\mo} and $\Lambda={\End_{A}}(M)^{op}$. Let $S_{i}$ be the top of the indecomposable projective $\Lambda$-module ${\Hom}_{A}(M,M_{i})$. We show the following result.

\begin{theorem}{\rm (Theorem 3.1)}

The Gorenstein projective dimension of $S_{i}$ is at most n if $M_{i}$ is a projective $A$-module. Otherwise the Gorenstein projective dimension of $S_{i}$ is equal to n+1.

\end{theorem}

Note that the dominant dimension of a minimal n-Auslander-Gorenstein algebra $\Lambda$ is at least n+1 and there exist projective-injective $\Lambda$-modules. We denote by prinj($\Lambda$) the full subcategory of projective-injective $\Lambda$-modules and we denote by GP$^{\leq n}(\Lambda)$ the full subcategory of $\Lambda$-modules whose Gorenstein projective dimensions are at most n and by sub$\Lambda$ the full subcategory of $\Lambda$-modules which can be cogenerated by a finite direct sum of $\Lambda$.

We prove that a minimal n-Auslander-Gorenstein algebra $\Lambda$ can be characterised by the relations between the Gorenstein projective dimensions of $\Lambda$-modules and their socles. We also refer to \cite{NRTZ} and \cite{PS}  for other characterisations of minimal n-Auslander-Gorenstein algebras.

\begin{theorem}{\rm (Theorem 4.1)}

Let $\Lambda$ be a $(n+1)$-Gorenstein algebra. Then $\Lambda$ is a minimal n-Auslander-Gorenstein algebra if and only if it satisfies {\rm{prinj}}$(\Lambda)$=$\{I\in \Lambda{\text{-}\mo}\mid I$ is injective and ${\Gpd(\soc}I)\leq n\}$.

\end{theorem}

\begin{theorem}{\rm (Theorem 4.3)}

Let $\Lambda$ be an Artin algebra. Then $\Lambda$ is a minimal n-Auslander-Gorenstein algebra if and only if it satisfies {\rm{GP}}$^{\leq n}(\Lambda)=\{N\in \Lambda\text{-}{\mo}\mid {\Gpd(\soc}N)\leq n\}={\rm{sub}}\Lambda$.

\end{theorem}

Theorem 2 and Theorem 3 generalise the main results in \cite{LZ} where the results are proved for the special case of Auslander algebras.

This paper is arranged as follows. In section 2, we fix the notions and recall some necessary facts. In section 3, we prove Theorem 1 and further results needed in section 4. Section 4 is devoted to the proof of Theorem 2 and Theorem 3.

\section{Preliminaries}

Throughout this paper, we consider basic Artin algebras. All modules are finitely generated left modules in case nothing is stated otherwise. For an algebra A, we denote by $A$-mod the category of finitely generated left A-modules. For an $A$-module $N$, Gpd$N$ $({\pd}N)$ is the Gorenstein projective (projective) dimension of $N$. We denote by $\Omega^{i}N$ $(\Omega^{-i}N)$ the $i$th syzygy (cosyzygy) of $N$ and by $\tau_{A}$ the Auslander-Reiten translation of A. The composition of morphisms $f:X\rightarrow Y$ and $g:Y\rightarrow Z$ is denoted by $gf: X\rightarrow Z$. We  follow the standard terminologies and notations used in the representation theory of algebras, see \cite{ASS} and \cite{ARS}. For background on Gorenstein homological algebra we refer to \cite{CXW}.

First we recall some results about Gorenstein projective dimensions.

\begin{lemma}{\rm \cite[Proposition 3.2.2]{CXW}}

Let $N$ be an $A$-module of finite Gorenstein projective dimension and $n$ be an integer. Then the following are equivalent:

$(1)$ {\Gpd}$N\leq n$,

$(2)$ ${\Ext}^{i}_{A}(N,L)=0$ for $i\geq n+1$ and any $A$-module $L$ of finite projective dimension,

$(3)$ ${\Ext}^{i}_{A}(N,A)=0$ for $i\geq n+1$,

$(4)$ ${\Ext}^{n+1}_{A}(N,L)=0$ for any $A$-module $L$ of finite projective dimension.

Consequently, {\Gpd}$N={\rm {sup}}\{i\geq0\mid {\Ext}^{i}_{A}(N,A)\neq0\}$.

\end{lemma}
We remark that in a Gorenstein algebra of self-injective dimension $n$, every module has finite Gorenstein projective dimension at most $n$.

The next lemma gives relations between Gorenstein projective dimensions of the three modules in a short exact sequence.

\begin{lemma}{\rm \cite[Corollary 3.2.4]{CXW}}

Let $0\rightarrow X \rightarrow Y \rightarrow Z\rightarrow 0$ be a short exact sequence in $A$-{\mo}. Then we have

$(1)$ {\Gpd}$Y\leq {\rm max}\{{\Gpd}X, {\Gpd}Z\}$,

$(2)$ {\Gpd}$X\leq {\rm max}\{{\Gpd}Y, {\Gpd}Z-1\}$,

$(3)$ {\Gpd}$Z\leq {\rm max}\{{\Gpd}Y, {\Gpd}X+1\}$.

\end{lemma}

Let $\mathcal{C}$ be a finite n-precluster tilting subcategory in $A$-mod and $M$ be a basic additive generator of $\mathcal{C}$, namely $\mathcal{C}=$add$M$. In \cite{IS}, Iyama and Solberg generalised the Auslander correspondence to the following higher version.

\begin{thm}{\rm \cite[Theorem 4.5]{IS}}

There is a bijection between equivalence classes of finite n-precluster tilting subcategories $\mathcal{C}$ and Morita-equivalence classes of minimal n-Auslander-Gorenstein algebras $\Lambda$, where the correspondence is given by $M\mapsto \Lambda={\End}_{A}(M)^{op}$.

\end{thm}

For an n-precluster tilting subcategory $\mathcal{C}$ in $A$-{\mo}, we have $\mathcal{C}^{\perp n-1}=^{\perp n-1}\mathcal{C}$ where
$\mathcal{C}^{\perp n-1}=\{X\in A\text{-}{\mo}\mid{\Ext}^{i}_{A}(\mathcal{C},X)=0 \, {\rm for}\, 0<i<n\}$ and $^{\perp n-1}\mathcal{C}$ is defined similarly. Let $\mathcal{Z}(\mathcal{C})=\mathcal{C}^{\perp n-1}=^{\perp n-1}\mathcal{C}$ and $\mathcal{U}(\mathcal{C})=\mathcal{Z}(\mathcal{C})/[\mathcal{C}]$. We know that $\mathcal{C}\subseteq \mathcal{Z}(\mathcal{C})$. Moreover, $\mathcal{Z}(\mathcal{C})$ has a structure of a Frobenius category and the projective-injective objects are precisely objects in $\mathcal{C}$. Thus $\mathcal{U}(\mathcal{C})$ is a triangulated category. Furthermore, if $\mathcal{C}$ is a finite n-precluster tilting subcategory with an additive generator $M$ and $\Lambda$={\End}$_{A}(M)^{op}$, we have

\begin{thm}{\rm \cite[Theorem 4.7]{IS}}

Let ${\rm GP}(\Lambda)$ be the subcategory of $\Lambda$-{\mo} consisting of all Gorenstein projective $\Lambda$-modules. Then $\mathcal{Z}(\mathcal{C})$ and ${\rm GP}(\Lambda)$ are equivalent via the functor {\Hom}$_{A}(M,-)$. Moreover {\Hom}$_{A}(M,-)$ induces a triangle equivalence between $\mathcal{U}(\mathcal{C})$ and ${\rm \underline{GP}}(\Lambda)$.

\end{thm}

Let $\mathcal{B}$ be a subcategory of $A$-mod. A complex $X^{\bullet}=(X^{i},d^{i})_{i\in \mathbb{Z}}$ of $A$-modules is called ${\Hom}_{A}(\mathcal{B,-})$-exact (${\Hom}_{A}(-,\mathcal{B})$-exact) if $X^{\bullet}$ is exact and for any $Y\in\mathcal{B}$, ${\Hom}_{A}(Y,X^{\bullet})$ $({\Hom}_{A}(X^{\bullet},Y))$ is also exact.

The next lemma shows that for any $A$-module $X$, there exists a ${\Hom}_{A}(\mathcal{C},-)$-exact $\mathcal{Z}(\mathcal{C})$-resolution (${\Hom}_{A}(-,\mathcal{C})$-exact $\mathcal{Z}(\mathcal{C})$-coresolution).

\begin{lemma}{\rm \cite[Corollary 3.16]{IS}}

Let $\mathcal{C}$ be an n-precluster tilting subcategory and X in $A$-{\mo}. Then

\noindent $(1)$ For any $0\leq i \leq n-1$, there exists a ${\Hom}_{A}(\mathcal{C},-)$-exact sequence
$$0\rightarrow C_{n-1}\xrightarrow{f_{n-1}}\cdots \xrightarrow{f_{i+2}}C_{i+1}\xrightarrow{f_{i+1}}Z_{i}\xrightarrow{f_{i}}C_{i-1}\xrightarrow{f_{i-1}}\cdots\xrightarrow{f_{1}}C_{0}\xrightarrow{f_{0}} X\rightarrow 0 $$
 with $Z_{i}$ in $\mathcal{Z}(\mathcal{C})$ and $C_{j}$ in $\mathcal{C}$ for any $j$. Moreover {\rm Im}$f_{j}$ is in $\mathcal{C}^{\perp j}$ for any $j$.

\noindent $(2)$ For any $0\leq i \leq n-1$, there exists a ${\Hom}_{A}(-,\mathcal{C})$-exact sequence
$$0\rightarrow X\xrightarrow{f^{0}}C^{0} \xrightarrow{f^{1}}\cdots \xrightarrow{f^{i-1}}C^{i-1}\xrightarrow{f^{i}}Z^{i}\xrightarrow{f^{i+1}}C^{i+1}\xrightarrow{f^{i+2}}\cdots\xrightarrow{f^{n-1}} C^{n-1}\rightarrow 0 $$
with $Z^{i}$ in $\mathcal{Z}(\mathcal{C})$ and $C^{j}$ in $\mathcal{C}$ for any $j$. Moreover {\rm Im}$f^{j}$ is in $^{\perp j}\mathcal{C}$ for any $j$.

\end{lemma}

In \cite{IS},  Iyama and Solberg introduced a higher Auslander-Reiten theory for n-precluster tilting subcategories.

\begin{definition}{\rm \cite[Definition 5.1]{IS}}

Let $A$ be an Artin algebra, $X$ and $Y$ be $A$-modules and $\eta$ be a non-zero element in ${\Ext^{n}_{A}(X,Y)}$.

$(1)$ We say that $\eta$ is a right n-fold almost split extension of $X$ if for any $A$-module $Z$ and a non-zero element $\xi$ in ${\Ext}^{n}_{A}(X,Z)$, there exists a morphism $f:Z\rightarrow Y$ such that $\eta =f\xi$.

$(2)$ We say that $\eta$ is a left n-fold almost split extension of $Y$ if for any $A$-module $Z$  and a non-zero element $\xi$  in ${\Ext}^{n}_{A}(Z,Y)$, there exists a morphism $g: X\rightarrow Z$ such that $\eta =\xi g$.

$(3)$ We say that $\eta$ is an n-fold almost split extension if it is a right n-fold almost split extension of $X$ and a left n-fold almost split extension of $Y$.

\end{definition}

Let $X$ and $Y$ be indecomposable $A$-modules. If $\eta \in{\Ext}^{n}_{A}(X,Y)$ is an n-fold almost split extension,  we have $Y=\tau_{n}X$ and $X=\tau^{-}_{n}Y$.

We denote by $\mathcal{P}(A)$ $(\mathcal{I}(A))$ the full subcategory of $A$-mod consisting of all projective (injective) modules. The following result shows that there exist n-fold almost split extensions in $\mathcal{Z}(\mathcal{C})$.

\begin{thm}{\rm \cite[Theorem 5.10]{IS}}

Let $\mathcal{C}$ be an n-precluster tilting subcategory in $A$-{\mo}, $X$ be an indecomposable module in $\mathcal{Z}(\mathcal{C})\setminus\mathcal{P}(A)$, and $Y:=\tau_{n}(X)$ be the corresponding indecomposable module in $\mathcal{Z}(\mathcal{C})\setminus\mathcal{I}(A)$.

$(1)$ For each $0\leq i \leq n-1$, an n-fold almost split extension in ${\Ext}^{n}_{A}(X,Y)$ can be represented as
$$0\rightarrow Y \rightarrow C_{n-1} \rightarrow \cdots \rightarrow C_{i+1} \rightarrow Z_{i} \rightarrow C_{i-1}
\rightarrow \cdots \rightarrow C_{0} \rightarrow X \rightarrow 0$$
with $Z_{i}$ in $\mathcal{Z}(\mathcal{C})$ and $C_{j}$ in $\mathcal{C}$ for any $j$.

$(2)$ The following sequences are exact.
\begin{equation*}
\begin{split}
0 \rightarrow {\Hom}_{A}(\mathcal{C},Y) \rightarrow {\Hom}_{A}(\mathcal{C},C_{n-1})\rightarrow \cdots \rightarrow {\Hom}_{A}(\mathcal{C},C_{i+1}) \rightarrow {\Hom}_{A}(\mathcal{C},Z_{i}) \\
\rightarrow {\Hom}_{A}(\mathcal{C},C_{i-1}) \rightarrow \cdots \rightarrow {\Hom}_{A}(\mathcal{C},C_{0})\rightarrow {\rm rad}_{A}(\mathcal{C},X) \rightarrow 0,
\end{split}
\end{equation*}
\begin{equation*}
\begin{split}
0 \rightarrow {\Hom}_{A}(X,\mathcal{C}) \rightarrow {\Hom}_{A}(C_{0},\mathcal{C})\rightarrow \cdots \rightarrow {\Hom}_{A}(C_{i-1}, \mathcal{C}) \rightarrow {\Hom}_{A}(Z_{i},\mathcal{C})\\
\rightarrow {\Hom}_{A}(C_{i+1},\mathcal{C}) \rightarrow \cdots \rightarrow {\Hom}_{A}(C_{n-1},\mathcal{C})\rightarrow {\rm rad}_{A}(Y,\mathcal{C}) \rightarrow 0.
\end{split}
\end{equation*}

$(3)$ If $X$ and $Y$ do not belong to $\mathcal{C}$, then the n-fold almost split extension in $(1)$ can be given as a Yoneda product of a minimal $\mathcal{Z}(\mathcal{C})$-resolution of $X$
$$0\rightarrow \Omega^{i}_{\mathcal{Z}(\mathcal{C})}(X)\rightarrow C_{i-1} \rightarrow \cdots C_{0} \rightarrow X \rightarrow 0,$$
an almost split sequence in $\mathcal{Z}(\mathcal{C})$
$$0 \rightarrow \Omega^{-(n-i-1)}_{\mathcal{Z}(\mathcal{C})}(Y) \rightarrow Z_{i} \rightarrow \Omega^{i}_{\mathcal{Z}(\mathcal{C})}(X)\rightarrow 0,$$
and a minimal $\mathcal{Z}(\mathcal{C})$-coresolution of $Y$
$$0\rightarrow Y \rightarrow C_{n-1}\rightarrow \cdots \rightarrow C_{i+1} \rightarrow \Omega ^{-(n-i-1)}_{\mathcal{Z}(\mathcal{C})}(Y)\rightarrow 0$$
where $\Omega^{k}_{\mathcal{Z}(\mathcal{C})}$ $(\Omega^{-l}_{\mathcal{Z}(\mathcal{C})})$ is the $k$th $(l$th$)$ syzygy (cosyzygy) with respect to the minimal $\mathcal{Z}(\mathcal{C})$-resolution $(\mathcal{Z}(\mathcal{C})$-coresolution$)$.

\end{thm}

\section{Gorenstein projective dimensions of modules over minimal Auslander-Gorenstein algebras}

In this section, we investigate the relations between the Gorenstein projective dimensions of modules over minimal Auslander-Gorenstein algebras and their socles.

Let $A$ be an Artin algebra and $\mathcal{C}$ be a finite n-precluster tilting subcategory in $A$-mod with a basic additive generator $M$. Then we have $\mathcal{C}=$add$M$ and $\Lambda={\End}_{A}(M)^{op}$ is a minimal n-Auslander-Gorenstein algebra. Since the socle of a $\Lambda$-module is the direct sum of its simple submodules, we first calculate the Gorenstein projective dimensions of all simple $\Lambda$-modules.

\begin{thm}

Let $M=\oplus^{t}_{i=1}M_{i}$ with $M_{i}$ indecomposable and $S_{i}$ be the top of the indecomposable projective $\Lambda$-module ${\Hom}_{A}(M,M_{i})$. Then we have

$(1)$ If $M_{i}$ is a projective $A$-module, then the Gorenstein projective dimension of $S_{i}$ is at most n,

$(2)$ If $M_{i}$ is a non-projective $A$-module, then the Gorenstein projective dimension of $S_{i}$ is equal to n+1.

\end{thm}

\begin{proof}

$(1)$ Assume $M_{i}$ is a projective $A$-module, then the inclusion $f:{\rad}M_{i}\rightarrow M_{i}$ is a right almost split morphism. We claim that {\rm Im}${\Hom}_{A}(M,f)={\rad}_{A}(M,M_{i})$.

Let $h\in {\rad}_{A}(M_{j},M_{i})$, then $h$ is not a retraction. Since $f$ is a right almost split morphism, there exists $g:M_{j}\rightarrow {\rad}M_{i}$ such that $h=fg\in$Im${\Hom}_{A}(M_{j},f)$. Thus ${\rad}_{A}(M_{j},M_{i})\subseteq$Im${\Hom}_{A}(M_{j},f)$. For the reverse inclusion, assume first $j\neq i$, then {\rad}$(M_{j},M_{i})={\Hom}_{A}(M_{j},M_{i})$ and clearly Im${\Hom}_{A}(M_{j},f)\subseteq$ ${\rad}_{A}(M_{j},M_{i})$.
If $j=i$, let $h\in$Im${\Hom}_{A}(M_{i},f)$, then there exists $g:M_{i}\rightarrow {\rad}M_{i}$ such that $h=fg$. $h$ is not an isomorphism otherwise $f$ is a retraction which contradicts the fact that $f$ is right almost split. Thus $h\in{\rad}_{A}(M_{i},M_{i})$. Then we get {\rm Im}${\Hom}_{A}(M,f)=\oplus^{t}_{j=1}{\rm{Im}\Hom}_{A}(M_{j},f)=\oplus^{t}_{j=1}{\rad}_{A}(M_{j},M_{i})={\rad}_{A}(M,M_{i})$
and a short exact sequence $$0 \rightarrow {\Hom}_{A}(M,{\rad}M_{i}) \rightarrow {\Hom}_{A}(M,M_{i}) \rightarrow S_{i} \rightarrow 0.$$
Now consider ${\rad}M_{i}$, by Lemma 2.5(1), there exists a long exact sequence $$0\rightarrow C_{n-1}\rightarrow\cdots \rightarrow C_{s+1}\rightarrow Z_{s}\rightarrow C_{s-1}\rightarrow \cdots\rightarrow C_{0}\rightarrow {\rad}M_{i}\rightarrow 0 $$
with $Z_{s}$ in $\mathcal{Z}(\mathcal{C})$ for some $0\leq s \leq n-1$ and $C_{j}$ in $\mathcal{C}$ for any $j$. Applying ${\Hom}_{A}(M,-)$, we get the following long exact sequence
\begin{equation*}
\begin{split}
0\rightarrow {\Hom}_{A}(M,C_{n-1})\rightarrow \cdots \rightarrow {\Hom}_{A}(M,C_{s+1}) \rightarrow {\Hom}_{A}(M,Z_{s}) \rightarrow \\
{\Hom}_{A}(M,C_{s-1}) \rightarrow \cdots {\Hom}_{A}(M,C_{0}) \rightarrow {\Hom}_{A}(M,{\rad}M_{i}) \rightarrow 0.
\end{split}
\end{equation*}
According to Theorem 2.4, we know that the following long exact sequence is a Gorenstein projective resolution of $S_{i}$.
\begin{equation*}
\begin{split}
0\rightarrow {\Hom}_{A}(M,C_{n-1})\rightarrow \cdots \rightarrow {\Hom}_{A}(M,C_{s+1}) \rightarrow {\Hom}_{A}(M,Z_{s}) \rightarrow \\
{\Hom}_{A}(M,C_{s-1}) \rightarrow \cdots {\Hom}_{A}(M,C_{0}) \rightarrow {\Hom}_{A}(M,M_{i}) \rightarrow S_{i}\rightarrow 0
\end{split}
\end{equation*}
Thus Gpd$S_{i}\leq n$.

$(2)$ Assume $M_{i}$ is not projective. By Theorem 2.7(1), the n-fold almost split extension in ${\Ext}^{n}_{A}(M_{i},\tau_{n}M_{i})$ can be represented as
$$0\rightarrow \tau_{n}M_{i} \xrightarrow{f} C_{n-1} \rightarrow \cdots \rightarrow C_{s+1} \rightarrow Z_{s} \rightarrow C_{s-1}
\rightarrow \cdots \rightarrow C_{0} \rightarrow M_{i} \rightarrow 0 \qquad (*)$$
with $Z_{s}$ in $\mathcal{Z}(\mathcal{C})$ for some $0\leq s \leq n-1$ and $C_{j}$ in $\mathcal{C}$ for any $j$.

Claim 1: $f:\tau_{n}M_{i} \rightarrow C_{n-1}$ is not a section. Otherwise there exists $g:C_{n-1}\rightarrow \tau_{n}M_{i}$ such that $gf={\rm id}_{\tau_{n}M_{i}}$. Since $\mathcal{C}$ is closed under $\tau_{n}$ and $M_{i}$ is indecomposable, $\tau_{n}M_{i}\in \mathcal{C}$ is also indecomposable. Applying ${\Hom}_{A}(-,\tau_{n}M_{i})$ to $(*)$, according to Theorem 2.7(2), the following sequence is exact.
\begin{equation*}
\begin{split}
0 \rightarrow {\Hom}_{A}(M_{i},\tau_{n}M_{i}) \rightarrow {\Hom}_{A}(C_{0},\tau_{n}M_{i})\rightarrow \cdots \rightarrow {\Hom}_{A}(C_{s-1}, \tau_{n}M_{i}) \rightarrow {\Hom}_{A}(Z_{s},\tau_{n}M_{i})\\
\rightarrow {\Hom}_{A}(C_{s+1},\tau_{n}M_{i}) \rightarrow \cdots \rightarrow {\Hom}_{A}(C_{n-1},\tau_{n}M_{i})\xrightarrow{{\Hom}_{A}(f,\tau_{n}M_{i})} {\rm rad}_{A}(\tau_{n}M_{i},\tau_{n}M_{i}) \rightarrow 0.
\end{split}
\end{equation*}
Then ${\Hom}_{A}(f,\tau_{n}M_{i})(g)=gf={\rm id}_{\tau_{n}M_{i}} \in {\rm rad}_{A}(\tau_{n}M_{i},\tau_{n}M_{i})$, a contradiction.

Applying ${\Hom}_{A}(M,-)$ to $(*)$, by Theorem 2.7(2) again, the following long exact sequence is exact.
\begin{equation*}
\begin{split}
0 \rightarrow {\Hom}_{A}(M,\tau_{n}M_{i}) \rightarrow {\Hom}_{A}(M,C_{n-1})\rightarrow \cdots \rightarrow {\Hom}_{A}(M,C_{s+1}) \rightarrow {\Hom}_{A}(M,Z_{s}) \\
\rightarrow {\Hom}_{A}(M,C_{s-1}) \rightarrow \cdots \rightarrow {\Hom}_{A}(M,C_{0})\rightarrow {\rm rad}_{A}(M,M_{i}) \rightarrow 0,
\end{split}
\end{equation*}
Then we get a Gorenstein projective resolution of $S_{i}$.
\begin{equation*}
\begin{split}
0 \rightarrow {\Hom}_{A}(M,\tau_{n}M_{i}) \xrightarrow{{\Hom}_{A}(M,f)} {\Hom}_{A}(M,C_{n-1})\rightarrow \cdots \rightarrow {\Hom}_{A}(M,C_{s+1}) \rightarrow {\Hom}_{A}(M,Z_{s}) \\
\rightarrow {\Hom}_{A}(M,C_{s-1}) \rightarrow \cdots \rightarrow {\Hom}_{A}(M,C_{0})\rightarrow {\Hom}_{A}(M,M_{i}) \rightarrow S_{i}\rightarrow 0 (**)
\end{split}
\end{equation*}

Claim 2: ${\Ext}^{n+1}_{\Lambda}(S_{i},\Lambda)\neq 0$. Otherwise suppose ${\Ext}^{n+1}_{\Lambda}(S_{i},\Lambda)= 0$. Note that ${\Hom}_{A}(M,Z_{s})\in ^{\perp}\Lambda$ since ${\Hom}_{A}(M,Z_{s})$ is Gorenstein projective. Applying ${\Hom}_{\Lambda}(-,\Lambda)$ to $(**)$, by dimension shifting, we get ${\Hom}_{A}(M,f)$ is a section. Since ${\Hom}_{A}(M,-)$ is an equivalence between $\mathcal{Z}(\mathcal{C})$ and ${\rm GP}(\Lambda)$, there exists $g:C_{n-1}\rightarrow \tau_{n}M_{i}$ such that ${\Hom}_{A}(M,g){\Hom}_{A}(M,f)={\Hom}_{A}(M,gf)={\rm id}_{{\Hom}_{A}(M,\tau_{n}M_{i})}$. Thus $gf={\rm id}_{\tau_{n}M_{i}}$ and $f$ is a section, which contradicts Claim 1.

Thus ${\Ext}^{n+1}_{\Lambda}(S_{i},\Lambda)\neq 0$ and by Lemma 2.1, ${\Gpd}S_{i}=n+1$.

\end{proof}

It is known that the socle of a $\Lambda$-module $N$ coincides with the socle of its injective envelope $I(N)$. Next we investigate the injective modules over the minimal n-Auslander-Gorenstein algebra $\Lambda$ and give a description of all the projective-injective $\Lambda$-modules.

\begin{thm}

Let $\Lambda$ be a minimal n-Auslander-Gorenstein algebra. Then an injective $\Lambda$-module $I$ is projective if and only if the Gorenstein projective dimension of its socle is at most n, that is {\rm{prinj}}$(\Lambda)$=$\{I\in \Lambda{\text{-}\mo}\mid I$ is injective and ${\Gpd(\soc}I)\leq n\}$.

\end{thm}

\begin{proof}

Consider the short exact sequence $0 \rightarrow {\soc}\Lambda \rightarrow \Lambda \rightarrow \Lambda/{\soc}\Lambda\rightarrow 0.$ If $\Lambda/{\soc}\Lambda$ is a Gorenstein projective $\Lambda$-module, so is ${\soc}\Lambda$. Otherwise ${\Gpd(\soc}\Lambda)={\Gpd}(\Lambda/{\soc}\Lambda)-1\leq n$ since $\Lambda$ is an $(n+1)$-Gorenstein algebra.

If $I$ is a projective-injective $\Lambda$-module, then ${\soc}I\in$add$({\soc}\Lambda)$. Thus ${\Gpd(\soc}I)\leq n$.

Conversely, suppose $I$ is an indecomposable injective $\Lambda$-module with ${\Gpd(\soc}I)\leq n$. Since ${\id}_{\Lambda}\Lambda\leq n+1 \leq {\dom}\Lambda$, there exists a minimal injective coresolution of $\Lambda$
$$0 \rightarrow \Lambda \rightarrow I_{0}\rightarrow I_{1} \rightarrow \cdots \rightarrow I_{n+1}\rightarrow 0$$
with $I_{j}$ projective for all $0\leq j \leq n$. ${\Gpd(\soc}I)=s\leq n$ implies that ${\Ext}^{s}_{\Lambda}({\soc}I,\Lambda)\neq 0$. Because of ${\Ext}^{s}_{\Lambda}({\soc}I,\Lambda)={\Hom}_{\Lambda}({\soc}I, I_{s})$, we know that ${\soc}I$ is contained in ${\soc}I_{s}$. Thus $I$ is a direct summand of $I_{s}$ and $I$ is projective since $I_{s}$ is a projective $\Lambda$-module.

\end{proof}

Since the minimal n-Auslander-Gorenstein algebra $\Lambda$ is either a self-injective algebra or an (n+1)-Gorenstein algebra, the Gorenstein projective dimensions of $\Lambda$-modules are at most n+1. Now we show the relations between the Gorenstein projective dimensions of $\Lambda$-modules and their socles.

\begin{thm}

Let $\Lambda$ be a minimal n-Auslander-Gorenstein algebra and $N$ be a $\Lambda$-module. Then the following are equivalent:

$(1)$ The Gorenstein projective dimension of $N$ is at most n,

$(2)$ The Gorenstein projective dimension of {\soc}$(N)$ is at most n,

$(3)$ $N$ is cogenerated by a finite direct sum of $\Lambda$.

That is {\rm{GP}}$^{\leq n}(\Lambda)=\{N\in \Lambda\text{-}{\mo}\mid {\Gpd(\soc}N)\leq n\}={\rm{sub}}\Lambda$.

\end{thm}

\begin{proof}

$(1)\Longrightarrow (2):$ Assume ${\Gpd}N\leq n$. Consider the short exact sequence $0 \rightarrow {\soc}N \rightarrow N\rightarrow N/{\soc}N \rightarrow 0.$ By Lemma 2.2,  we have ${\Gpd}({\soc}N)\leq {\rm max}\{{\Gpd}N, {\Gpd}(N/{\soc}N)-1\}\leq n$.

$(2)\Longrightarrow(3):$ If ${\Gpd}({\soc}N)\leq n$, we know that ${\Gpd}({\soc}I(N))\leq n$ where $I(N)$ is the injective envelope of $N$. By Theorem 3.2, $I(N)$ is projective. Thus $N$ is cogenerated by a projective $\Lambda$-module.

$(3)\Longrightarrow(1):$  $N$ is cogenerated by a projective $\Lambda$-module $P$ and we can get a short exact sequence $0 \rightarrow N \rightarrow P \rightarrow P/N \rightarrow 0$. If $P/N$ is a Gorenstein projective module, so is $N$. Otherwise ${\Gpd}N={\Gpd}(P/N)-1\leq n$.

\end{proof}

It follows that a $\Lambda$-module $N$ has the highest Gorenstein projective dimension n+1 if and only if its socle ${\soc}N$ has the highest Gorenstein projective dimension n+1. Thus the Gorenstein projective dimensions of $\Lambda$-modules with highest Gorenstein projective dimension are determined by the Gorenstein projective dimensions of their socles.

\section{Characterizations of minimal Auslander-Gorenstein algebras}

In \cite{PS}, Pressland and Sauter proved a new characterization of minimal n-Auslander-Gorenstein algebras by using shifted and coshifted tilting modules. Minimal 1-Auslander-Gorenstein algebras can be characterised by the existence of a special tilting-cotilting module or that the Gorenstein projective modules category is an abelian category, see \cite{NRTZ} and \cite{KF} for details. In this section, we prove that minimal n-Auslander-Gorenstein algebras can be characterised by the relations between the Gorenstein projective dimensions of modules and their socles.

Our first characterisation of minimal n-Auslander-Gorenstein algebras is given in terms of projective-injective modules.

\begin{thm}

Let $\Lambda$ be a (n+1)-Gorenstein algebra. Then $\Lambda$ is a minimal n-Auslander-Gorenstein algebra if and only if it satisfies {\rm{prinj}}$(\Lambda)$=$\{I\in \Lambda{\text{-}\mo}\mid I$ is injective and ${\Gpd(\soc}I)\leq n\}$.

\end{thm}

\begin{proof}

If $\Lambda$ is a minimal n-Auslander-Gorenstein algebra, it satisfies {\rm{prinj}}$(\Lambda)$=$\{I\in \Lambda{\text{-}\mo}\mid I$ is injective and ${\Gpd(\soc}I)\leq n\}$ by Theorem 3.2.

Conversely, since $\Lambda$ is an (n+1)-Gorenstein algebra, it satisfies ${\id}_{\Lambda}\Lambda = n+1$. We only need to show ${\dom}\,\Lambda \geq n+1$. Consider the minimal injective coresolution of $\Lambda$
$$0 \rightarrow \Lambda \xrightarrow{f_{0}} I_{0} \xrightarrow{f_{1}} I_{1} \rightarrow \cdots \xrightarrow{f_{n}} I_{n} \xrightarrow{f_{n+1}} I_{n+1} \rightarrow 0.$$
Note that $I_{0}$ is the injective envelope of $\Lambda$ and ${\Gpd(\soc}\Lambda)\leq n$, we have ${\Gpd(\soc}I_{0})={\Gpd(\soc}\Lambda)\leq n$. Thus $I_{0}$ is a projective $\Lambda$-module by our assumption and then ${\Gpd(\im}f_{1})\leq 1$. Now consider the short exact sequence $$0 \rightarrow {\soc(\im}f_{1}) \rightarrow {\im}f_{1} \rightarrow {\im}f_{1}/{\soc(\im}f_{1}) \rightarrow 0.$$ According to Lemma 2.2, ${\Gpd(\soc(\im}f_{1}))\leq {\rm max}\{{\Gpd(\im}f_{1}),{\Gpd(\im}f_{1}/{\soc(\im}f_{1}))-1\}\leq n$. Because $I_{1}$ is the injective envelope of ${\im}f_{1}$, ${\Gpd(\soc}I_{1})\leq n$ and $I_{1}$ is projective. Then we have ${\Gpd(\im}f_{2})\leq 2.$ Continuing this procedure, we can get that $I_{j}$ is projective for $j=0,1,\ldots,n.$ Thus ${\dom}\,\Lambda \geq n+1$. This completes our proof.

\end{proof}

Recall that an Artin algebra $\Gamma$ is called an n-Auslander algebra if it satisfies ${\gl}\,\Gamma \leq n+1 \leq {\dom}\,\Gamma$. Here {\gl} denotes global dimension.

For algebras of finite global dimension, Gorenstein projective modules coincide with projective modules. Specialising Theorem 4.1 to algebras of finite global dimension, we obtain:

\begin{corollary}

Let $\Gamma$ be an Artin algebra with global dimension n+1. Then $\Gamma$ is an n-Auslander algebra if and only if it satisfies {\rm{prinj}}$(\Gamma)$=$\{I\in \Gamma{\text{-}\mo}\mid I$ is injective and ${\pd(\soc}I)\leq n\}$.

\end{corollary}

Now we prove another characterization of minimal n-Auslander-Gorenstein algebras.

\begin{thm}

Let $\Lambda$ be an Artin algebra. Then $\Lambda$ is a minimal n-Auslander-Gorenstein algebra if and only if it satisfies {\rm{GP}}$^{\leq n}(\Lambda)=\{N\in \Lambda\text{-}{\mo}\mid {\Gpd(\soc}N)\leq n\}={\rm{sub}}\Lambda$.

\end{thm}

\begin{proof}

If $\Lambda$ is a minimal n-Auslander-Gorenstein algebra, it satisfies {\rm{GP}}$^{\leq n}(\Lambda)=\{N\in \Lambda\text{-}{\mo}\mid {\Gpd(\soc}N)\leq n\}={\rm{sub}}\Lambda$ by Theorem 3.3.

Conversely, let $X$ be a $\Lambda$-module and $P(X)$ be its projective cover. Then we obtain a short exact sequence $0 \rightarrow Y \rightarrow P(X) \rightarrow X \rightarrow 0.$ Since $Y \in {\rm sub}\Lambda$, by assumption, we have ${\Gpd}Y \leq n$ and thus ${\Gpd}X \leq n+1$. This implies that $\Lambda$ is an m-Gorenstein algebra with ${\id}_{\Lambda}\Lambda= m \leq n+1$.

First assume ${\id}_{\Lambda}\Lambda < n+1$. Let $I$ be an indecomposable injective $\Lambda$-module, we know that ${\Gpd}I\leq {\id}_{\Lambda}\Lambda \leq n$. Then $I \in {\rm sub}\Lambda$ and $I$ is projective. Thus $\Lambda$ is a self-injective algebra.

If ${\id}_{\Lambda}\Lambda = n+1$, there exists a minimal injective coresolution of $\Lambda$ $$0 \rightarrow \Lambda \xrightarrow{f_{0}} I_{0} \xrightarrow{f_{1}} I_{1} \rightarrow \cdots \xrightarrow{f_{n}} I_{n} \xrightarrow{f_{n+1}} I_{n+1} \rightarrow 0.$$

By assumption, $\Lambda \in {\rm sub}\Lambda$ implies ${\Gpd(\soc}\Lambda)\leq n$. Since $I_{0}$ is the injective envelope of $\Lambda$, we have ${\Gpd(\soc}I_{0})\leq n$. By assumption again, $I_{0}\in {\rm sub}\Lambda$ and thus $I_{0}$ is a projective $\Lambda$-module. Then we get ${\Gpd(\im}f_{1})\leq 1$. So ${\im}f_{1}\in {\rm GP}^{\leq n}(\Lambda)$ and ${\Gpd(\soc(\im}f_{1}))\leq n$. Note that $I_{1}$ is the injective envelope of ${\im}f_{1}$, we have $I_{1}$ is projective and ${\Gpd(\im}f_{2})\leq 2$. Continuing this procedure, we get all $I_{j} (0 \leq j \leq n)$ are projective and thus ${\dom}\,\Lambda \geq n+1$. It follows that $\Lambda$ is a minimal n-Auslander-Gorenstein algebra.

\end{proof}

Immediately, we have the following corollary.

\begin{corollary}

Let $\Gamma$ be an Artin algebra and ${\rm P}^{\leq n}(\Gamma)$ be the subcategory of $\Gamma$-modules whose projective dimensions are at most n. Then $\Gamma$ is an n-Auslander algebra if and only if it satisfies {\rm{P}}$^{\leq n}(\Gamma)=\{N\in \Gamma\text{-}{\mo}\mid {\pd(\soc}N)\leq n\}={\rm{sub}}\Gamma$.

\end{corollary}

Our results in this section generalise the main theorems in \cite{LZ}, where the theorems were proved for the special case of Auslander algebras.

We give two counterexamples to questions that might arise that we found using the GAP-package QPA, see \cite{QPA}. We remark that we use right modules in those counterexamples since QPA always uses right modules. Both examples are Nakayama algebras. We refer to \cite{ARS} and \cite{Mar3} for the basics and homological algebra of Nakayama algebras. One might ask whether in a minimal n-Auslander-Gorenstein algebra of infinite global dimension, a module $N$ has projective dimension n+1 if and only if its socle has projective dimension n+1. The following example shows that this does not hold in general.

\begin{example} \rm

Let $\Lambda$ be the Nakayama algebra with Kupisch series [3, 3, 4] and simple modules numbered from 1 to 3. Then $\Lambda$ is a minimal 1-Auslander-Gorenstein algebra with infinite global dimension. The right module $N:=  e_{1} \Lambda /e_{1} J^2$ has projective dimension 2, but its socle has infinite projective dimension.

\end{example}

One might ask whether in a higher Auslander algebra it holds that a module $N$ has projective dimension 2 if and only if its socle has projective dimension 2, as a generalisation of Theorem B in \cite{LZ} from Auslander algebras to higher Auslander algebras. The next example shows that this is not true in general.

\begin{example} \rm

Let $\Gamma$ be the Nakayama algebra with Kupisch series [3, 3, 3, 3, 2, 1] and simple modules numbered from 1 to 6. Then $\Gamma$ is a 2-Auslander algebra. Let $N$ be the right module $e_{3} \Gamma /e_{3} J^2$. Then $N$ has projective dimension 2, but its socle has projective dimension 1.

\end{example}

\section*{Acknowledgements}
We are thankful to Steffen Koenig for useful comments.
This paper was written when the first author was visiting University of Stuttgart from October 2017 to September 2018. He would like to thank Prof. Steffen Koenig and the rest of the IAZ for their warm hospitality and kind help. This work is supported by the National Natural Science Foundation of China (11671230, 11601274, 11371165).

\newpage

Shen Li: School of Mathematics, Shandong University, PR China

\emph{E-mail address}: fbljs603@163.com

Ren{\'{e}} Marczinzik: Institute of algebra and number theory, University of Stuttgart, Germany

\emph{E-mail address}: marczire@mathematik.uni-stuttgart.de

Shunhua Zhang: School of Mathematics, Shandong University, PR China

\emph{E-mail address}: shzhang@sdu.edu.cn

\end{document}